\documentclass[a4paper,10pt]{amsart}
\usepackage{amsmath}
\usepackage{amssymb}
\usepackage[arrow, matrix, curve]{xy}
\usepackage{lmodern,dsfont}

\newtheorem{theorem}{Theorem}[section]
\newtheorem{proposition}[theorem]{Proposition}

\newtheorem{lemma}[theorem]{Lemma}
\newtheorem{definition}[theorem]{Definition}
\theoremstyle{remark}
\newtheorem{remark}{\bf Remark}

\newenvironment{proof_thm1}[1][Proof of Theorem~\ref{main_thm}:]{\begin{trivlist}
\item[\hskip \labelsep {\bfseries #1}]}{\end{trivlist}}
\newenvironment{proof_prop_trans}[1][Proof of Proposition~\ref{prop_transformation}:]{\begin{trivlist}
\item[\hskip \labelsep {\bfseries #1}]}{\end{trivlist}}

\title{A Root Parametrized Differential Equation for the Special Linear Group}
\author{Matthias Sei\ss}
\address{Matthias Sei\ss \\ Mathematisches Institut Universit\"at Heidelberg \\
Im Neuenheimer Feld 288 \\
69120 Heidelberg \\ Germany}
\begin{document}
\maketitle
\begin{abstract}
Let $C \langle \textit{\textbf{t}} \rangle$ be the differential field generated by $l$ differential indeterminates 
$\textit{\textbf{t}}=(t_1, \dots, t_l)$ over an algebraically closed field $C$ of characteristic zero. 
In this article we present an explicit linear parameter differential equation over 
$C \langle \textit{\textbf{t}} \rangle$ with differential Galois group $\mathrm{SL}_{l+1}(C)$
and show that it is a generic equation in the following sense: If $F$ is an algebraically closed differential field
with constants $C$ and $E/F$ is a Picard-Vessiot extension with differential Galois group 
$H(C) \subseteq \mathrm{SL}_{l+1}(C)$, then a specialization of our equation 
defines a Picard-Vessiot extension differentially isomorphic to $E/F$. 
\end{abstract}
\section{Introduction}
For a differential field $F$ of characteristic zero with algebraically closed field of constants $C$ and a linear
algebraic group $G$ over $C$ the so-called inverse problem in differential Galois theory asks 
whether the group $G$ can be 
realized as a differential Galois group of a Picard-Vessiot extension $E/F$ for a linear differential 
equation over $F$. 
A solution to the inverse problem is known for the 
field of rational functions $C(z)$ with standard derivation $\frac{d}{dz}$ 
and was proved by J. Hartmann (\cite{Hart}) in 2002. In the years before 
important partial successes in this setting were achieved by several researchers. 
In 1994, M. Singer showed in \cite{Singer} that a large 
class of groups can be realized as differential Galois groups over $C(z)$ 
carrying over results from the classical setting when $C$ is the field of 
complex numbers. A constructive approach uses the Lie algebra of the group. 
With the strategy to choose an appropriate element
in the Lie algebra for the definition of a linear differential equation, 
A. Magid obtained first results at the end of \cite{Magid}.
In \cite{M/S} C. Mitschi and M. Singer used bound criteria for the differential Galois group
and showed that all connected groups occur as differential Galois groups over $C(z)$.
\\
In this article we present a method for the realization of the classical groups as differential
Galois groups over the differential field $C\langle t_1,\dots, t_l \rangle$ where $C\langle t_1,\dots, t_l \rangle$
is the differential field which is differentially generated by $l$ differential indeterminates
$\textit{\textbf{t}}=(t_1, \dots , t_l)$ over $C$ and $l$ denotes the Lie rank of the group. 
We explain how our method 
uses the geometric structure of the Lie group and we exhibit  
exemplarily the proofs for the special linear group $\mathrm{SL}_{l+1}$. Our method is constructive and yields explicit linear 
parameter differential equations.
The main results of this article are Theorem~\ref{main_thm} 
which contains the linear parameter differential equation for $\mathrm{SL}_{l+1}$
and Theorem~\ref{thm_gen} which states that this equation 
is a quasi-generic differential equation.
Nice parameter differential equations for the groups of type $B_l$, $C_l$, $D_l$ and $G_2$ ($l=2$) 
can be found in the last section.
\begin{theorem}
\label{main_thm}
The linear parameter differential equation
\begin{equation*}
L(y,\textbf{t})= y^{(l+1)} - \sum\nolimits_{i=1}^{l} t_i y^{(i-1)}=0
\end{equation*}
has $\mathrm{SL}_{l+1}(C)$ as differential Galois group over $C\langle \textit{\textbf{t}} \rangle$.
\end{theorem}
Let $G$ be one of the classical groups. The key tools for a realization of $G$ are the Lie algebra $\mathfrak{g}$
and bound criteria for the differential Galois group. More precisely,  
the idea is to choose an appropriate element $A( \textit{\textbf{t}})$ from the Lie algebra  
$\mathfrak{g}(C\langle \textit{\textbf{t}} \rangle)$  
for the definition of a matrix differential equation $\partial(\textit{\textbf{y}})= A( \textit{\textbf{t}})\textit{\textbf{y}}$  
such that we have enough information to show that the differential 
Galois group can not be smaller or larger than $G(C)$. The strategy is to construct an $A( \textit{\textbf{t}})$ 
such that it represents well the geometric structure of $G$.  
Given a root space decomposition of $\mathfrak{g}$, we choose $A( \textit{\textbf{t}})$ 
such that is has non-zero constant 
components in the root spaces belonging to the negative of the simple roots and 
such that the differential indeterminates 
$t_1, \dots , t_l$ parameterize $l$ root spaces which correspond to $l$ specific positive roots 
of height equal to the exponents of the root sytem.
We can then apply bound criteria for the differential Galois group of 
 $\partial(\textit{\textbf{y}})= A( \textit{\textbf{t}})\textit{\textbf{y}}$ and 
using structure theory, we can show that the upper and lower bound coincide in $G(C)$ 
for our choice of  $A( \textit{\textbf{t}})$.  
\\ 
Let $G$ be a linear algebraic group over $C$ and let $C\langle s_1, \dots, s_n \rangle$ 
be the  differential field  which is generated by $n$ differential indeterminates 
$\textit{\textbf{s}}=(s_1, \dots, s_n)$ over $C$.
The generic inverse problem asks whether $G$ can be realized
over $C\langle \textit{\textbf{s}} \rangle$ in such a way that 
every Picard-Vessiot extension with differential Galois group $G(C)$ over a differential field $F$
with constants $C$ can be obtained be specializing the indeterminates to elements of $F$.
In the literature there are three main approaches for a solution to the generic inverse problem (\cite{Gold}, 
\cite{Miller} and \cite{LJAL}) with different definitions of genericity. 
In \cite{Gold} Goldman uses ideas similar to E. Noether's for polynomial equations in classical Galois theory
to compute generic equations for some specific groups $G \subset \mathrm{GL}_n$ 
including $\mathrm{SL}_{l+1}$.
He starts with a differential field generated by $n$   
differential indeterminates $y_1, \hdots , y_n$ and considers then the 
fixed field under the action of
$G(C)$ which is induced by matrix multiplication on the Wronskian $\mathrm{W}(y_1, \hdots , y_n)$. 
Goldman's construction yields a differential equation $L(y,\textit{\textbf{u}})$ where the coefficients are elements 
of the fixed field $C\langle \textit{\textbf{u}} \rangle$ which is generated by $n$ differentially independent
elements $\textit{\textbf{u}}=(u_1, \dots, u_n)$.
He shows that such an equation for a group $G$ satisfies the following property: 
If $E/F$ is a Picard-Vessiot
extension for a linear differential equation $L(y)$ with differential Galois group a subgroup of $G$ over any differential field $F$ with constants $C$,
then there is a specialization $\sigma: \textit{\textbf{u}} \mapsto \textit{\textbf{f}}$ such that $L(y,\sigma(\textit{\textbf{u}}))=L(y)$ 
where $\textit{\textbf{f}}=(f_1,\dots,f_n)$ with $f_i  \in F$.
Goldman's definition of a generic equation (see \cite{Gold}) 
 is more general than Definition~\ref{def_generic} below.
It does not require any additional property of the differential field $F$ and the 
extensions are obtained directly.
In 1970, J. Miller studied in \cite{Miller} {\it differentially Hilbertian differential fields} and solved the generic inverse problem
for some specific groups. 
In \cite{LJAL} L. Juan and A. Ledet 
pursued another method for the determination of generic equations. 
Their method is based on Kolchin's Structure Theorem which describes all possible
Picard-Vessiot extensions as function fields of irreducible $G$-torsors. 
In the case of $\mathrm{SO}_n$ their method is well applicable and yields a generic matrix 
differential equation with $\frac{1}{2}(n+2)(n-1)$ parameters.\\
We will show that the linear parameter differential equation 
in Theorem~\ref{main_thm} for $\mathrm{SL}_{l+1}(C)$ satisfies the following definition of a {\it quasi-generic} equation.
\begin{definition}\label{def_generic}
 Let $G$ be a linear algebraic group over $C$ and let $L(y,\textit{\textbf{s}})$ be a linear parameter 
 differential equation over $C\langle \textit{\textbf{s}} \rangle$. The linear differential equation 
 $L(y,\textit{\textbf{s}})$ will be called a {\it quasi-generic} differential equation for $G$, if the following
 conditions are satisfied:
 \begin{enumerate}
 \item The differential Galois group of $L(y,\textit{\textbf{s}})$ over $C\langle \textit{\textbf{s}} \rangle$ is $G(C)$.
\item If $F$ is an algebraically closed differential field with constants $C$ and $E/F$ is a Picard-Vessiot extension
with differential Galois group $H(C) \subseteq G(C)$, then there is a specialization 
  $\sigma: \textit{\textbf{s}} \mapsto \textit{\textbf{f}} $
 such that 
$L(y,\sigma(\textit{\textbf{s}}))$ defines a Picard-Vessiot extension which is differentially isomorphic to $E/F$
where $\textit{\textbf{f}}=(f_1,\dots,f_n)$ with $f_i  \in F$.
\item For a differential field $F$ with constants $C$ and any specialization 
$\sigma:\textit{\textbf{s}} \mapsto \textit{\textbf{f}}$ 
the differential Galois group of a Picard-Vessiot extension 
for $L(y,\sigma(\textit{\textbf{s}}))$ is a subgroup of $G(C)$ 
where $\textit{\textbf{f}}=(f_1,\dots,f_n)$ with $f_i  \in F$.
\end{enumerate}
\end{definition}

\begin{theorem}\label{thm_gen}
The equation in Theorem~\ref{main_thm} is a quasi-generic differential equation for $\mathrm{SL}_{l+1}(C)$ 
\end{theorem}
Thus, Theorem~\ref{thm_gen} gives a differential analogue for the group $\mathrm{SL}_{l+1}(C)$ of the Kummer equations for 
regular cyclic extensions in classical Galois theory. In particular, our equations for $\mathrm{SL}_{l+1}(C)$ are relatively simple 
differential equations which become generic after a suitable algebraic extension.\\
\section{Bounds for the differential Galois group}
We recall some basic definitions from differential Galois theory. 
Let  $F$ be an ordinary differential field of characteristic zero with an algebraically closed field of constants $C$ and derivation $\partial$. 
A linear differential equation over $F$ is an equation of the form  
$\partial(\textit{\textbf{y}})=A\textit{\textbf{y}}$ where $A \in \mathrm{M}_n(F)$. Here, 
$\mathrm{M}_n(F)$ denotes 
the set of all $n \times n$-matrix with coefficients in $F$.
A Picard-Vessiot ring $R$ for a linear differential equation $\partial(\textit{\textbf{y}})=A\textit{\textbf{y}}$ 
over $F$ is a 
differential ring which satisfies the following three properties:
\begin{enumerate}
 \item The ring $R$ is a simple differential ring, that is $R$ has no non-trivial differential ideals (ideals which are stable under the derivation).
\item There exists $Y \in \mathrm{GL}_n(R)$ such that $\partial(Y)=AY$.
\item The ring $R$ is generated as a ring by the entries $Y_{ij}$ of $Y$ and $\mathrm{det}(Y)^{-1}$ over $F$.
\end{enumerate}
The matrix $Y$ in (2) is called a fundamental solution matrix for $\partial(\textit{\textbf{y}})=A\textit{\textbf{y}}$.
A Picard-Vessiot field for $\partial(\textit{\textbf{y}})=A\textit{\textbf{y}}$ over $F$ is a differential field $E$ which 
is the field of fractions of a Picard-Vessiot ring for the equation. The differential Galois group $G$
of a linear differential equation $\partial(\textit{\textbf{y}})=A\textit{\textbf{y}}$ over $F$ is 
the group of all differential $F$-automorphisms of $E$ and it has a representation as 
a linear algebraic group. For a detailed introduction to differential Galois theory, 
we refer to the books \cite{CresHaj}, \cite{Magid} and \cite{P/S}. Throughout this article $C$ denotes an
algebraically closed field of characteristic zero. \\

Let $G$ be one of the classical groups of Lie rank $l$ over $C$ and let $F$ be a differential field with constants $C$.
The main idea for the realization of $G(C)$ is the definition of an appropriate differential structure on a finite 
dimensional $F$-vector space $M$ such that the differential Galois group in a sense to be defined by $M$ can not be larger or smaller than 
$G(C)$. The differential structure on $M$ is defined by a matrix differential equation $\partial(\textit{\textbf{y}})=A\textit{\textbf{y}}$ where 
$A \in \mathrm{M}_n(F)$ and $n$ is the dimension of a representation of $G(C)$. As in \cite{M/S} we apply an upper
and lower bound criterion to $A$ and choose $A \in \mathrm{M}_n(F)$ such that both bounds coincide. 
The main ingredient
for a successful choice of $A$ is the Lie algebra $\mathfrak{g}$ of $G$.\\
An upper bound criterion for the differential Galois group is given by the following Proposition 
(see \cite{P/S}, Proposition 1.31, (1)) which was first proven by Kovacic.
\begin{proposition}\label{Prop_upper}
 Let $H$ be a connected linear algebraic group over $C$ and let $A \in \mathfrak{h}(F)$. Then
 the differential Galois group $G(C)$ of the differential equation $\partial(\textbf{y})=A \textbf{y}$ 
 is contained (up to conjugation) in $H(C)$.
\end{proposition}
Let $R$ be a Picard-Vessiot ring for $\partial(\textit{\textbf{y}})=A \textit{\textbf{y}}$ with Galois group $G$. Then the affine group scheme 
$\mathrm{Spec}(R)=\mathcal{Z}$ over $F$ is a $G$-torsor (see \cite{P/S}, Theorem 1.28). If $\mathcal{Z}$ has an $F$-rational 
point, that is $\mathcal{Z}$ is the trivial torsor, then Proposition~\ref{Prop_upper} has a partial converse. 
\begin{proposition}\label{Prop_lower}
 Let $R$ be a Picard-Vessiot ring for $\partial(\textit{\textbf{y}}) =A \textit{\textbf{y}}$ over $F$ with connected differential Galois group $G(C)$ and
 let $\mathcal{Z}$ be the associated torsor. Let $H(C) \supset G(C)$ be a connected 
 linear algebraic group with $A \in \mathfrak{h}(F)$. If $\mathcal{Z}$ is the trivial torsor, then
 there exists $B \in H(F)$ such that $\partial(B)B^{-1} + B A B^{-1}$ is an element of 
 $\mathfrak{g}(F)$.
\end{proposition}
For a proof see \cite{P/S}, Corollary 1.32.\\
The condition $\mathcal{Z}$ being the trivial torsor in Proposition~\ref{Prop_lower} is automatically satisfied if the 
cohomological dimension of $F$ is at most one (see \cite{Serr}, Chapter III, 2.4).
From \cite{Serr}, Chapter II, 3.3 b), we know that this is true for $F_2:=C(z)$, i.e. the function field with standard derivation $\frac{d}{dz}$.
In this setting C. Mitschi and M. Singer found a way in \cite{M/S}
to apply successfully Proposition~\ref{Prop_upper} and~\ref{Prop_lower} for a realization of a connected 
semisimple group $G(C)$.\\
In our situation, i.e. for the differential field $F_1:=C\langle t_1, \dots, t_l \rangle$, we have no information about whether 
the condition of Proposition~\ref{Prop_lower} is satisfied or not and we can therefore not use it as a lower bound criterion.
But we can apply it in an indirect way. 
To this purpose denote by $C \{ \textit{\textbf{t}} \}$ the differential ring which is differentially generated by the $l$ differential indeterminates
$\textit{\textbf{t}}=(t_1, \dots,t_l)$ over $C$ and
let $\partial(\textit{\textbf{y}})=A(\textit{\textbf{t}})\textit{\textbf{y}}$ be a matrix differential equation with defining matrix
$A(\textit{\textbf{t}})  \in C \{ \textit{\textbf{t}} \}^{n \times n}$.
We consider now a surjective specialization $\sigma : R_1 \rightarrow R_2$, where $R_1$, $R_2$ are
the differential subrings $R_1:=C \{ \textit{\textbf{t}} \} \subset F_1$ and $R_2 := C[z] \subset F_2$.
Then $\sigma$ yields a new 
differential equation $\partial(\textit{\textbf{y}})=A(\sigma(\textit{\textbf{t}}))\textit{\textbf{y}}$ over $F_2$.
Intuitively, we would now expect that the differential Galois group of the specialized equation
$\partial(\textit{\textbf{y}})=A(\sigma(\textit{\textbf{t}})) \textit{\textbf{y}}$ is contained in the differential
Galois group of the original equation. Indeed, we introduce the so-called specialization bound.
\begin{theorem}\label{specialization_bound}
Suppose the defining matrix $A(\textit{\textbf{t}})$ satisfies $A(\textit{\textbf{t}}) \in R_{1}^{n \times n}$.
Then the differential Galois group of the specialized equation  $\partial(\textit{\textbf{y}}) =A (\sigma ( \textit{\textbf{t}}))\textit{\textbf{y}}$ over $F_2$ is a 
subgroup of the differential Galois group for $A(\textit{\textbf{t}})$ over $F_1$. 
\end{theorem}
Our proof of Theorem~\ref{specialization_bound} is quite involved and is given in \cite{Seiss}, Theorem 4.3.
\section{The method and results from the theory of algebraic groups}
In this section, we describe 
the choice of the defining matrix $A(\textit{\textbf{t}})$ in $\mathfrak{g}(R_1)$ 
and the strategy to show that the upper and lower bound coincide for it. We 
recall some structure theory about semisimple linear algebraic groups.\\
Let $\Phi$ be the root system of $\mathfrak{g}(C)$ and denote by  $\Delta= \{ \alpha_1, \dots , \alpha_l \}$
a basis of $\Phi$. We write $\Phi^+$ for the set of positive roots of $\Phi$ and $\Phi^-$ for 
the negative roots, respectively. Let 
\begin{equation*}
  \mathfrak{g}(C) = \mathfrak{h}(C) \oplus \bigoplus_{\alpha \in \Phi} \mathfrak{g}_{\alpha}(C)
\end{equation*}
be a Cartan decomposition for $\mathfrak{g}(C)$ with Cartan algebra $\mathfrak{h}(C)$ in diagonal form
and one-dimensional root spaces $ \mathfrak{g}_{\alpha}(C)$ for the roots $\alpha \in \Phi$.
Let us denote by $H_{\alpha} \in \mathfrak{h}$ the co-root for a root $\alpha \in \Phi$, 
meaning that $H_{\alpha}$ is given by $H_{\alpha}= 2 \alpha/ (\alpha,\alpha)$, where $\alpha$ on the right-hand side 
is identified with an element of $\frak{h}$ by the relation $\alpha(H)=(\alpha, H)$ for all $H \in \frak{h}$ and 
$(\cdot , \cdot)$ denotes the Killing form.
Then, we can choose elements $X_{\alpha} \in \mathfrak{g}_{\alpha}$ for each root $\alpha \in \Phi$ 
which satisfy the following properties: 
\begin{align*}
&[ X_{\alpha} , X_{-\alpha} ] = H_{\alpha}, \\
&[X_{\alpha},X_{\beta} ]= \pm (r+1) \ X_{\alpha + \beta}, 
\end{align*}
where for $\alpha,  \beta \in \Phi$ the integer $r$ is the largest such that $\beta - r \alpha$ is a root.
The elements
\begin{equation*}
 \{ X_{\alpha}, \ H_{\alpha_i} \mid \alpha \in \Phi,\alpha_i \in \Delta \}
\end{equation*}
form a so-called Chevalley basis associated to the above Cartan decomposition for $\mathfrak{g}(C)$. 
We fix such a basis for $\mathfrak{g}(C)$.   
Note that the co-roots $H_{\alpha_i}$ corresponding to the simple roots generate the Cartan algebra 
$\mathfrak{h}(C)$ and for $\alpha \in \Phi$ the element $X_{\alpha}$ clearly forms a basis of the root space 
$\mathfrak{g}_{\alpha}(C)$. 
We denote in the following by $\mathfrak{n}^+ = \sum_{\alpha \in \Phi^+} \mathfrak{g}_{\alpha}$ 
the maximal nilpotent 
subalgebra of $\mathfrak{g}$ and by $\mathfrak{n}^-$ the maximal nilpotent
subalgebra defined by the negative roots, respectively. We write further $\mathfrak{b}^+$ 
(resp. $\mathfrak{b}^-$) for the maximal solvable subalgebra of $\mathfrak{g}$ with the property 
$\mathfrak{b}^+ = \mathfrak{h} + \mathfrak{n}^+$ 
(resp. $\mathfrak{b}^- = \mathfrak{h} + \mathfrak{n}^-$) and we denote by $B^+$ (resp. $B^-$) 
the Borel subgroup of $G$ with Lie algebra $\mathfrak{b}^+$ (resp. $\mathfrak{b}^-$). 
Finally, let $X \in \mathfrak{g}$ and let $\mathfrak{s}$ be a subspace of $\mathfrak{g}$. 
Then we call the affine subspace $X+ \mathfrak{s}$ a plane of $\mathfrak{g}$. 
\\
In the following, we describe our method for the realization of a classical group $G$ 
as a differential Galois group. 
We explain the choice of the defining matrix $A(\textit{\textbf{t}})$. 
For the simple roots $\alpha_i \in \Delta$, we define the matrix 
$A_{\Delta}^- := \sum_{\alpha_i \in \Delta} X_{-\alpha_i}$ and 
$A_{\Delta}^+ := \sum_{\alpha_i \in \Delta} X_{\alpha_i}$ accordingly. 
It is then possible to show that there are $l$ roots 
$\gamma_i \in \Phi^+$ ($1 \leq i \leq l$), which are of height equal to the $l$ exponents of the root system of $G$ 
and whose choice depends on further properties of $\Phi$, such that the matrix
\begin{equation*}
 A(\textit{\textbf{t}}):= A_{\Delta}^- + \sum_{i=1}^l t_i X_{ \gamma_i} \in \mathfrak{g}(R_1)
\end{equation*}
does not lie in any subalgebra of $\mathfrak{g}(F_1)$ and covers by specialization 
a wide range of gauge-equivalent matrices (see below).
In addition, the differential equation $\partial(\textit{\textbf{y}})=A(\textit{\textbf{t}}) \textit{\textbf{y}}$ 
has a canonical cyclic vector which induces easily a linear parameter equation of a nice shape. 
We want to mention that we may interchange the role of the positive and negative roots 
in the definition of $A(\textit{\textbf{t}})$ to obtain a more convenient shape of the defining matrix.\\
For a successful application of the specialization bound, we need a differential equation 
$\partial(\textit{\textbf{y}})=\bar{A}\textit{\textbf{y}}$ over $F_2$ 
which is a specialization of the parameter equation 
$\partial(\textit{\textbf{y}})= A (\textit{\textbf{t}}) \textit{\textbf{y}}$ 
and has a known differential Galois group. 
Unfortunately, we have no information about the Picard-Vessiot extensions defined by the equations which are directly
available as specializations of $A(\textit{\textbf{t}})$.
As a solution we consider matrices which are gauge-equivalent to specializations of $A(\textit{\textbf{t}})$. 
Here, two matrices $A$ and $\tilde{A}$ are called gauge-equivalent over a differential field $F$ if
\begin{equation*}
 BAB^{-1} + \partial(B)B^{-1}= \tilde{A}
\end{equation*}
for some $B \in \mathrm{GL}_n(F)$.
It is possible to describe a sufficiently large set of equations which are gauge-equivalent to specializations
of $A(\textit{\textbf{t}})$ using the geometric structure of $G$ and the choice of the roots in the 
definition of $A(\textit{\textbf{t}})$. 
To be more precise, we can show that every element $A$ in the plane 
\begin{equation}
\label{subspace}
 A \in A_{\Delta}^+ + \mathfrak{b}^-(F) 
\end{equation}
is gauge-equivalent to a specialization of $A(\textit{\textbf{t}})$. 
To this purpose let us consider the adjoint action
\begin{equation*}
 \mathrm{Ad}(B): \mathfrak{g} \rightarrow \mathfrak{g}, \ X \mapsto B X B^{-1} \ 
 \mathrm{for} \ B \in G(C)
\end{equation*}  
and the logarithmic derivative $l \delta$, which is defined by
 \begin{equation*}
  l \delta: \mathrm{GL}_n(F) \rightarrow \mathfrak{gl}_n(F), \ X \mapsto \partial_{F}(X)X^{-1}.
 \end{equation*}
Obviously, we can decompose the gauge transformation of $A$ into the sum of the two maps
$\mathrm{Ad}(B)(A)$ and $l\delta(B)$, i.e. we have
 \begin{equation*}
  BAB^{-1} + \partial(B)B^{-1} = \mathrm{Ad}(B)(A)+l \delta(B).
 \end{equation*}
In order to get a better grasp of the gauge transformation, we need more information about the
images of the two maps. To begin with we study the adjoint action.  
For $X \in \mathfrak{g}$ we denote by $\mathrm{ad}(X):\mathfrak{g} \rightarrow \mathfrak{g}$
the endomorphism of $\mathfrak{g}$ defined by sending $Y \in \mathfrak{g}$ to $\mathrm{ad(X)}(Y) = [X,Y]$. 
Then, for $X \in \mathfrak{g}$ nilpotent the exponential of $\mathrm{ad(X)}$, 
\begin{equation*}
\mathrm{exp}(\mathrm{ad}(X)) = \sum_{j \geq 0} \frac{1}{j!} \mathrm{ad}(X)^j ,
\end{equation*}
is an automorphism of $\mathfrak{g}$. In fact, for $\beta \in \Phi$ and $x \in F$, $ x \ \mathrm{ad(X_{\beta})}$ is 
a nilpotent endomorphism and the effect of the automorphism $x \ \mathrm{exp}(\mathrm{ad}(X_{\beta}))$
on elements of a Chevalley basis can be described by the root system. Now let
\begin{equation*}
\mathrm{exp}: \mathfrak{g}_{\beta}  \rightarrow U_{\beta}, 
\ X_{\beta} \mapsto \sum_{j \geq 0} \frac{1}{j!}  X_{\beta}^j
\end{equation*}
be the exponential map from the root space $\mathfrak{g}_{\beta}$ to the root group $U_{\beta}$ of $G$. 
Then the relation
\begin{equation*}
\mathrm{Ad}(\mathrm{exp}(xX_{\beta}))=\mathrm{exp}(x \ \mathrm{ad}(X_{\beta})) 
\end{equation*}
shows that it is possible to describe the adjoint action of a root group element 
$u_{\beta}(x)=\mathrm{exp}(xX_{\beta})$ on a Chevalley basis by the root system. 
The explicit formulae are given in the following remark.
 \begin{remark}\label{remark3}
  For $\alpha$, $\beta \in \Phi$ linearly independent let
 $\alpha - r \beta, \dots ,\alpha + q \beta$, for $r, q \in \mathds{N}$, be the $\beta$-string through $\alpha$ and let 
 $\langle \alpha, \beta \rangle$ be the Cartan integer. We define 
 $c_{\beta,\alpha,0}:=1$ and $c_{\beta,\alpha,i}:= \pm \binom{r+i}{i}$. Then, we have
  \begin{eqnarray*}
   \mathrm{Ad}(u_{\beta}(x))(X_{\alpha}) & =& \sum\nolimits_{i=0}^q c_{\beta, \alpha,i} x^i X_{\alpha + i \beta}, \\
   \mathrm{Ad}(u_{\beta}(x))(H_{\alpha}) &=& H_{\alpha} - \langle \alpha, \beta \rangle  X_{\beta} , \\
 \mathrm{Ad}(u_{\beta}(x))(X_{-\beta}) &=& X_{-\beta} + x H_{\beta} - x^2  X_{\beta}.
\end{eqnarray*}
 \end{remark}
 Finally, we look at the logarithmic derivative. Remark~\ref{remark4} below allows us to describe the image of the elements of the root groups under the logarithmic derivative 
  during the differential transformation of $A$ in terms of the roots.
\begin{remark}\label{remark4}
 Let $G \subset \mathrm{GL}_n$ be a linear algebraic group. Then the restriction of $l \delta$ to $G$ maps
 $G(F)$ to its Lie algebra $\mathfrak{g}(F)$, i.e. we have
 \begin{equation*}
  l \delta \mid_{G}: G(F) \rightarrow \mathfrak{g}(F).
 \end{equation*} 
\end{remark}
 A proof can be found in \cite{Kov}.\\
At this point, we want to note that N. Elkies refers in \cite{Elkies} exactly to the subspace in (\ref{subspace}).
More precisely, he uses the subspace $A_{\Delta}^- + \mathfrak{b}^+$ 
to define a {\it subvariety} $\mathcal{X}$ of the flag manifold 
$G/ B^+$ and proposes it as a differential analogue 
of the Deligne-Lusztig variety.\\
For a successful application of Theorem~\ref{specialization_bound}, we need a matrix differential equation 
$\partial(\textit{\textbf{y}}) = \bar{A}\textit{\textbf{y}}$ over $R_2$ which has $G(C)$ as differential Galois 
group and which satisfies $\bar{A} \in A_{\Delta}^- + \mathfrak{b}^+(R_2)$. Such an equation yields a 
 variant of a result from C. Mitschi and M. Singer which can be found in \cite{M/S}. 
 The difference of the original version to Proposition~\ref{mod_Mitch_Singer} below is that we modified the choice of the matrix $A_0$. 
\begin{proposition}\label{mod_Mitch_Singer}
 Let $G$ be a connected semisimple linear algebraic group and set $A_0 = \sum_{\alpha_i \in \Delta} 
 (X_{\alpha_i} + X_{- \alpha_i})$. Then there exists $A_1 \in \mathfrak{h}(C)$ such that the differential equation
 $\partial(\textit{\textbf{y}})=(A_0 + A_1 z)\textit{\textbf{y}}$ over $C(z)$ has $G$ as differential Galois group.
\end{proposition}
\begin{proof}
The strategy of the proof is to show that we can choose $A_1 \in \mathfrak{h}(C)$ such that the 
differential Galois group $G'$ of $\partial(\textit{\textbf{y}})=(A_0 + A_1 z)\textit{\textbf{y}}$ is a subgroup of $G$ 
and such that $G'$ is not equal to any proper subgroup of $G$. The first property is guaranteed by Proposition~\ref{Prop_lower}. 
Since the defining matrix $(A_0 + A_1 z)$ is for any choice 
$A_1 \in \mathfrak{h}(C)$ an element of the Lie algebra $\mathfrak{g}$, Proposition~\ref{Prop_lower} implies that $G'$ is a subgroup of $G$.\\
To show the second property more work is needed. A key ingredient to prove that 
$G'$ is not a proper subgroup of $G$ 
is a Chevalley module. This is a faithful representation $\rho: G \rightarrow \mathrm{GL}(W)$ 
with the property that $\rho(G)$ 
leaves no line in $W$ invariant, but any proper connected closed subgroup of $\rho(G)$ has an 
invariant one-dimensional subspace in $W$. 
Thus, a Chevalley module helps us to distinguish the group $G$ from its connected closed proper subgroups. 
From \cite{P/S}, Lemma 11.32 we obtain that the differential Galois group $G'$ 
is connected and Lemma 11.34 in \cite{P/S} guarantees that for a 
connected semisimple linear algebraic group $G$ a Chevalley module $W$ exists.
\\
Let $\rho: G \rightarrow \mathrm{GL}(W)$ be now a Chevalley module for $G$.
Then there is an induced injective morphism of Lie algebras
 $d\rho : \mathfrak{g}(C) \rightarrow \mathrm{End}(W)$, where we 
omit in the following the symbols $\rho$ (resp. $d \rho$ ) when we mean the action of $G$ 
(resp. $\mathfrak{g}$) on $W$. 
From the action of $\mathfrak{h}$ on $W$, we obtain a decomposition of 
$W= \bigoplus_{\lambda \in \Lambda} W_{\lambda}$ 
into finitely many weight spaces $W_{\lambda}$ for a finite number of weights 
$\lambda \in \Lambda  \subset \mathfrak{h}^*$, where 
$\mathfrak{h}^*$ denotes the dual space of $\mathfrak{h}$. Let us denote by $\Delta^{\pm}$ the set of all simple roots and their 
negatives, that is $\Delta^{\pm}= \Delta \cup \{ - \alpha_i \mid \alpha_i \in \Delta \} $.
We choose now $A_1 \in \mathfrak{h}$ such that it satisfies the following three properties:
\begin{enumerate}
	\item[(a)] The $\alpha(A_1)$ are non-zero and distinct for the roots $\alpha \in \Delta^{\pm}$.
	\item[(b)] The $\lambda(A_1)$ are non-zero and distinct for the non-zero weights 
	$\lambda$ of the representation $d \rho$.
	\item [(c)]All eigenvalues of $\sum_{\alpha \in \Delta^{\pm}} \frac{1}{\alpha(A_1)}X_{-\alpha}X_{\alpha}$ 
	which lie in $\mathbb{Z}$ are zero.
\end{enumerate}
The roots and the weights are linear combinations of the basis elements $H_{\alpha_i}^*$ of $\mathfrak{h}^*$, the dual basis 
for the basis $\{ H_{\alpha_i} \mid 1\leq i \leq l \}$ of $\mathfrak{h}$. 
Let $\hat{C}$ be a finite extension of $\mathbb{Q}$
containing these coefficients. Since $C$ is algebraically closed, there is an infinite $\hat{C}$-basis of $C$ 
and we can choose the entries of $A_1$ to be 
distinct basis elements. Then $A_1$ satisfies the first two conditions. If $A_1$ does not yet fulfill the third property, then a suitable multiple does.\\
Let $\lambda \in \Lambda$ be an arbitrary weight. Then $e:= \lambda(A_1)$ is an eigenvalue of $A_1$ with eigenspace $W_e$. We can now write each element $w \in W$ 
as a sum $w = \sum w_e$ of eigenvectors $w_e \in W_e$ for different eigenvalues $e$. 
For $\alpha \in \Delta^{\pm}$ and an eigenspace $W_e$ with eigenvalue $e$ we obtain from \cite{Hum}, Lemma 20.1, 
that $X_{\alpha} W_e \subset W_{\alpha(A_1) + e}$. 
Then the distinct values $\alpha(A_1)$ for the roots $\alpha \in \Delta^{\pm}$ imply that 
\begin{equation*}
 A_0 W_e \subset \bigoplus_{d \neq e} W_{d}.
\end{equation*}
Let us assume that $G'$ is a proper subgroup of $G$.
Since $W$ is a Chevalley module, $G'$ fixes then a line $\langle w \rangle_{C(z)}$ with $w \in W$, $w \neq 0$ in $W$
and this line is also stabilized by $\mathfrak{g}'$. 
Further, Proposition~\ref{Prop_lower} yields that there exists $B \in G(C(z))$ such that 
\begin{equation}
\label{eq_1_singer}
\tilde{A}:=B (A_0 + A_1 z)B^{-1} + (\frac{d}{dz} B )B^{-1} \in \mathfrak{g}'(C(z))
\end{equation}  
and in combination with the above we conclude that $\tilde{A}$ satisfies $\tilde{A}w = c w $ for a suitable $c \in C(z)$.
 For $\tilde{w}:=B^{-1}w \in C(z)\otimes W$ and $A:=(A_0 + A_1 z)$ we compute 
 with~(\ref{eq_1_singer}) and the relation $\frac{d}{dz}B^{-1} = -B^{-1} (\frac{d}{dz} B)B^{-1}$ the following:
\begin{eqnarray*}
A\tilde{w} &=& B^{-1}BAB^{-1}w = B^{-1} (\tilde{A} - (\frac{d}{dz} B )B^{-1})w \\
&=& B^{-1} \tilde{A}w - B^{-1} (\frac{d}{dz} B)B^{-1}w = c B^{-1} w + (\frac{d}{dz} B^{-1} )w = c \tilde{w} + \frac{d}{dz}\tilde{w}.
\end{eqnarray*}
The derivation $\frac{d}{dz}$ on $C(z)\otimes W$ is defined by 
$\frac{d}{dz}(f \otimes v)= ( \frac{d}{dz}f )\otimes v $. 
If $\tilde{w} \notin C[z] \otimes W$ we can multiply $w$ by the common denominator of the entries 
in $\tilde{w}$ and 
we can therefore assume without loss of generality that $\tilde{w} \in C[z]\otimes W$. We obtain with suitable 
$c_0,c_1 \in C$ the equation
\begin{equation}\label{eq_2_singer}
 \left( (A_0 + A_1 z) -\frac{d}{dz}  \right) \tilde{w} = (c_0 + c_1 z) \tilde{w}.
\end{equation}
Let $\tilde{w}= w_m z^m + ... + w_1 z + w_0$ with $w_i \in W$ and $w_m \neq 0$. 
Comparing the coefficients of $z^{m+1},z^{m},z^{m-1}$ in Equation~(\ref{eq_2_singer}) we get 
\begin{eqnarray}
A_1 w_m & = &  c_1 w_m, \label{eq_3_singer}  \\ 
A_0 w_{m} + A_1 w_{m-1} & = & c_0 w_m + c_1 w_{m-1}, \label{eq_4_singer} \\
A_0 w_{m-1} + A_1 w_{m-2} - m w_m& = & c_0 w_{m-1} + c_1 w_{m-2} \label{eq_5_singer}.
\end{eqnarray} 
In the following, for a vector $v \in W$ and an eigenvalue $d$, we denote by $(v)^{(d)}$ the component of $v$ in the eigenspace $W_d$. 
Equation~(\ref{eq_3_singer}) implies that $w_m$ is an eigenvector of $A_1$ with eigenvalue $e:= c_1 $ and lies in the eigenspace $W_e$. 
With this notation Equation~(\ref{eq_4_singer}) is equivalent to
\begin{equation}\label{eq_6.1_singer} 
 A_0 w_{m} + (A_1 - e) w_{m-1}  =  c_0 w_m \in W_e .
\end{equation}
The relations $A_0 w_{m} \subset \bigoplus_{d\neq e} W_d$ and $((A_1 - e) w_{m-1})^{(e)}=0$ 
show that the left hand side of (\ref{eq_6.1_singer}) has no component in the eigenspace $W_e$ and, therefore, we have 
that $c_0=0$. This leaves us with the following equation:
\begin{equation}\label{eq_6_singer}
 \sum_{d \neq e} (e - A_1) (w_{m-1})^{(d)} = \sum_{d \neq e} (A_0 w_{m})^{(d)}.
\end{equation}
Using Equation~(\ref{eq_6_singer}) we compute 
\begin{eqnarray*}
 w_{m-1} &=&  \sum_{d} (w_{m-1})^{(d)} = (w_{m-1})^{(e)} + \sum_{d \neq e} \frac{1}{e-d}(A_0 w_{m})^{(d)} \\
 &=& (w_{m-1})^{(e)} + \sum_{d \neq e} \frac{1}{e-d}(\sum_ {\alpha \in \Delta^{\pm}} X_{\alpha} w_{m})^{(d)}. 
\end{eqnarray*}
For $\alpha \in \Delta^{\pm}$ the relation $ X_{\alpha} W_{e} \subset W_{\alpha(A_1)+e}$ shows that 
\begin{equation*}
(\sum_ {\alpha \in \Delta^{\pm}} X_{\alpha} w_{m})^{(d)} =0
\end{equation*}
for all $\alpha(A_1)+e\neq d$ and, therefore, the above expression for $w_{m-1}$ can be simplified to
\begin{equation}\label{eq_7_singer}
 w_{m-1} =  (w_{m-1})^{(e)} -  \sum_ {\alpha \in \Delta^{\pm}} \frac{1}{\alpha(A_1)} X_{\alpha} w_{m}.
\end{equation}
Using the result that $c_0=0$, Equation~(\ref{eq_5_singer}) can be rewritten into
\begin{equation}\label{eq_8_singer}
 (e-A_1)w_{m-2} = A_0 w_{m-1} - m w_m.
\end{equation}
Since the left-hand side of~(\ref{eq_8_singer}) is contained in the subspace $\bigoplus_{d \neq e} W_d$, 
we obtain that
\begin{equation}\label{eq_9_singer}
  (A_0 w_{m-1} - m w_m)^{(e)} =0.
\end{equation}
Now, we substitute $w_{m-1}$ in~(\ref{eq_9_singer}) by the right-hand side of
Equation~(\ref{eq_7_singer}).  
In order to distinguish between the roots, 
we write $A_0 = \sum_{ \alpha' \in \Delta^{\pm}}X_{  \alpha'} $. We get
\begin{eqnarray*}
0 &=& \left(A_0( \sum_ {\alpha \in \Delta^{\pm}} \frac{1}{\alpha(A_1)} X_{\alpha} w_{m})\right)^{(e)} + (A_0 w_{m-1}^{(e)})^{(e)} - m w_{m} \\
 &=& \sum_ {\alpha \in \Delta^{\pm}} \frac{1}{\alpha(A_1)} \left(\sum_{\alpha' \in \Delta^{\pm}} X_{\alpha'} X_{\alpha} w_{m}\right)^{(e)} + 
 (\sum_{\alpha' \in \Delta^{\pm}} X_{\alpha'}w_{m-1}^{(e)})^{(e)} - m w_{m}.
\end{eqnarray*}
Since $X_{\alpha'} X_{\alpha} w_{m} \in W_{e + \alpha + \alpha'}$ and $X_{\alpha'}w_{m-1}^{(e)} \in \bigoplus_{d\neq e } W_d$, the above expression reduces to
\begin{equation*}
 ( \sum_ {\alpha \in \Delta^{\pm}} \frac{1}{\alpha(A_1)}  X_{-\alpha} X_{\alpha} ) w_{m}  - m w_{m}=0.
\end{equation*}
We conclude that $w_m$ is an eigenvector of the operator $\sum_ {\alpha \in \Delta^{\pm}} \frac{1}{\alpha(A_1)}  X_{-\alpha} X_{\alpha}$ 
with eigenvalue $-m$. Then by Condition (3) we have that $m=0$ and $w=w_0 \in W$.
This leaves us with the equation 
\begin{equation}\label{eq_10_singer}
(A_0 + z A_1)w_0 = c_1 z w_0. 
\end{equation}
Comparing the coefficients in Equation~(\ref{eq_10_singer}) yields $A_1w_0 =c_1 w_0$ and $A_0 w_0 = 0$. 
Thus, the one-dimensional subspace $\langle w_0 \rangle_C$ is invariant under $A_0$ and $A_1$. 
Therefore, it is also invariant under scalar multiples, sums and bracket products of $A_0$ and $A_1$. 
In the last step, we show that $A_0$ and $A_1$ generate the whole Lie algebra $\mathfrak{g}$. 
To this purpose we construct, for each $\alpha \in \Delta$, polynomials $P_{\alpha}(T),P_{-\alpha}(T) \in C[T]$ such that 
\begin{equation*}
P_{\pm \alpha}( \mathrm{ad} A_1)(A_0) = X_{\pm \alpha}.
\end{equation*}
To simplify the notation, we denote the negative simple roots by 
\begin{equation*}
 \alpha_{l+1}:= -\alpha_1,\ \dots, \ \alpha_{2l} =-\alpha_l .
\end{equation*}
For $i \in  \left\{1,...,2l \right\}$ we will show that there exist solutions $p_{i,j} \in C$ such that 
\begin{equation} 
\label{eq_11_singer}
X_{\alpha_i} = \sum_{j=1}^{2l} p_{i, j} \mathrm{ad}^{j}(A_1)(A_0).
\end{equation}
Equation~(\ref{eq_11_singer}) is equivalent to
\begin{equation*}
X_{\alpha_i} = \sum_{j=1}^{2l} \sum_{k=1}^{2l} p_{i, j} \alpha_{k}(A_1)^j X_{\alpha_k} =\sum_{k=1}^{2l} \left( \sum_{j=1}^{2l} p_{i, j} \alpha_{k}(A_1)^j \right) X_{\alpha_k}.  
\end{equation*} 
This is equivalent to show that for $1 \leq i \leq 2l$ there exist solutions of the following linear systems of equations:
\begin{equation} 
\label{eq_12_singer}
\left(
\begin{matrix}
\alpha_{1}(A_1) & \alpha_{1}(A_1)^2 & \cdots & \alpha_{1}(A_1)^{2l}    \\
\alpha_{2}(A_1) & \alpha_{2}(A_1)^2 & \cdots & \alpha_{2}(A_1)^{2l}    \\
\vdots     &              &        & \vdots             \\
\alpha_{2l}(A_1) & \alpha_{2l}(A_1)^2 & \cdots & \alpha_{2l}(A_1)^{2l} \\
\end{matrix}
\right)\cdot{
\left(
\begin{matrix}
p_{i, 1}     \\
p_{i, 2}     \\
\vdots       \\
p_{i, 2l}    \\
\end{matrix}
\right)}
= e_i
\end{equation}
where $e_i$ denotes the $i$-th unit vector. 
Since by Condition (a) all $\alpha_{i}(A_1)\neq 0$, 
the determinant of the matrix in~(\ref{eq_12_singer}) is non-zero if and only if
 the well-known \textit{Vandermonde determinant} for  $\alpha_{i}(A_1)$ ($1 \leq i \leq 2l$) is non-zero. 
Thus by Condition (a) 
 the determinant of the matrix in~(\ref{eq_12_singer}) 
 is non-zero  
and, therefore, there exist solutions $p_{ij} \in C$ such that Equation~(\ref{eq_11_singer}) holds.  
Thus, we can express the matrices $X_{\pm \alpha}$ ($\alpha \in \Delta$) in terms of linear combinations
of powers of bracket products in $A_0$ and $A_1$.
Since the matrices $\left\{ X_{\pm \alpha}\right\}_{\alpha \in \Delta}$ generate $\mathfrak{g}$,
we obtain that 
$A_0$ and $A_1$ also generate $\mathfrak{g}$.\\
The line $\langle w_0 \rangle_C$ is left invariant by $A_0$ and $A_1$ and, therefore, $\mathfrak{g}$ leaves 
this line invariant. 
Since $G$ is connected, we conclude that $G$ has also $\langle w_0 \rangle_C$ 
as an invariant one-dimensional subspace. 
But this contradicts the properties of a Chevalley module. 
\end{proof}
\section{A linear parameter differential equation for $\mathrm{SL}_{l+1}(C)$}
In the previous section, we have seen that a key ingredient for the realization of a classical group 
$G$ by our method 
is its geometric structure. For this reason, 
the proofs for the different classical groups are very similar and we present 
in this section exemplarily the proof for the group $G$ of type $A_l$, i.e. 
the special linear group $\mathrm{SL}_{l+1}(C)$. \\
An important object for the realization of 
one of the classical groups by our method is its root system. 
It is well-known that the root system of $\mathrm{SL}_{l+1}(C)$ 
is of type $A_l$. Let $\epsilon_1, \dots , \epsilon_{l+1}$ be
the standard orthonormal basis of $\mathbb{R}^{l+1}$ with respect to the usual inner product 
$(\cdot,\cdot)$ and denote by $I$ the $\mathbb{Z}$-span of this basis elements.  
 Let $I'=I \cap E$, where $E$ is the subspace of $\mathbb{R}^{l+1}$ orthogonal to the vector 
 $\epsilon_1 + \dots + \epsilon_{l+1}$. Then the root system $\Phi$ of type $A_l$ consists of the vectors 
 $\alpha \in I'$ with $(\alpha,\alpha)=2$, i.e. we have 
\begin{equation*}
\Phi= \{ \epsilon_i - \epsilon_j \mid 1\leq i,j \leq l+1 \}.
\end{equation*} 
 The elements $\alpha_i=\epsilon_i - \epsilon_{i+1}$
($1 \leq i \leq l$) are obviously independent and if $i<j$ we can write $\epsilon_i - \epsilon_j$ as
 $\epsilon_i - \epsilon_j = \alpha_i + \dots + \alpha_{j-1}$. This shows that 
 $\Delta = \{ \alpha_1, \dots , \alpha_l \}$ is a basis of $\Phi$
and we conclude that with respect to $\Delta$ the positive roots are 
\begin{equation*} 
  \Phi^+ = \{ \alpha_s + \hdots + \alpha_t \mid 1 \leq s \leq t \leq l \} ,
\end{equation*}
from which we obtain the negative roots $\Phi^- = \{ -\alpha \mid \alpha \in \Phi^+ \}$ 
by simply changing all signs. We want to note that for the following it is helpful to keep the 
shapes of the roots in mind.\\
The Dynkin diagram for $A_l$ shows that we can decompose $\Phi$ for $1 \leq k \leq l$ in subsystems
$\Phi_k$ of type $A_k$ with basis 
\begin{equation*}
 \Delta_k = \{ \alpha_{l-k+1}, \dots , \alpha_l \}.
\end{equation*}
Since $\Phi_k$ is now a root system of type $A_k$, 
it has a unique root of maximal height (see \cite{Hum}, 10.4, Lemma A). We denote this root by $\gamma_k$.
Then for $0 \leq k \leq l-1$, we define
\begin{equation*}
 \Gamma_k := \{ \gamma_i \mid \gamma_i \ \mathrm{is} \ \mathrm{the} \ \mathrm{maximal} \ \mathrm{root} \ \mathrm{of}
 \ \Phi_i^+ \ \mathrm{for} \ k+1 \leq i \leq l  \}
\end{equation*}
as the set of maximal roots of the descending chain of subsystems $\Phi= \Phi_l \supseteq \dots \supseteq \Phi_{k+1}$. 
To complete the definition of $\Gamma_k$ for all $0\leq k \leq l$ we define $\Gamma_l := \emptyset$ and 
we write shortly $\Gamma$ for $\Gamma_0$. 
\begin{remark} \label{remark1}
From the shapes of the roots in $\Phi^+$, we deduce that for $k \in \{ 1, \dots , l \}$  
the set $\Phi_k^+ \setminus \Phi_{k-1}^+$ consists of the roots
\begin{equation*}
\Phi_k^+ \setminus \Phi_{k-1}^+ =\{ \alpha_{l-k+1} + \dots + \alpha_{l-k+m} \mid 1 \leq m \leq k \}
\end{equation*}
and that $\Phi^+$ is the disjoint union of all $\Phi_k^+ \setminus \Phi_{k-1}^+$, where $\Phi_0^+$ 
is defined as the empty set. 
We conclude that for any element $m$ of $\{1, \dots ,k \}$ there is a unique root  
$\alpha \in \Phi_k^+ \setminus \Phi_{k-1}^+$ such that $\mathrm{ht}(\alpha)=m$.
A formal proof uses two inductions, i.e. an induction on the subsystems $\Phi_k$ and an inner induction on the height $m$
of the roots in $\Phi_k^+ \setminus \Phi_{k-1}^+$. 
\end{remark}
\begin{remark} \label{remark2}
Suppose $k \in \{ 1, \dots , l \}$ and $m \in \{1, \dots ,k \}$. Then Remark~\ref{remark1} implies that 
for a root $\alpha \in \Phi_k^+ \setminus (\Phi_{k-1}^+ \cup \{ \gamma_k \})$
with $\mathrm{ht}(\alpha)=m$ there exists a unique simple root $\alpha_s \in \Delta$ such that $\beta := \alpha + \alpha_s
\in \Phi_k^+ \setminus \Phi_{k-1}^+$. In particular, if $\beta - \alpha_t$ is a root for some $\alpha_t \in \Delta$,
then either $\beta - \alpha_t= \alpha$ or $\beta - \alpha_t \in \Phi_{k-1}^+$.
\end{remark}
For the determination of the linear differential equation in Theorem~\ref{main_thm}, we need an 
explicit Cartan decomposition of $\mathfrak{sl}_{l+1}$, the Lie Algebra of $\mathrm{SL}_{l+1}$, 
and a Chevalley basis according to this decomposition.
It is well-known that $\mathfrak{sl}_{l+1}$ is the set of all $(l+1)\times (l+1)$-matrices with trace zero.
Let $\mathfrak{h} \subset \mathfrak{sl}_{l+1}$ be the subalgebra of all diagonal matrices.
Then $\mathfrak{h}$ is a Cartan algebra of $\mathfrak{sl}_{l+1}$.
Now let $H=(h_1, \dots, h_{l+1}) \in \mathfrak{h}$ and denote by $E_{i,j}$ the 
$(l+1)\times (l+1)$-matrix with entry $1$ at position $(i,j)$ and $0$ elsewhere. 
Then the equation 
\begin{equation*}
 [H, E_{i,j}]= (h_i - h_j) E_{i,j}
\end{equation*}
shows that the matrix $E_{i,j}$ generates the root space $\mathfrak{sl}_{\epsilon_i - \epsilon_j}$. 
Note that for $1 \leq i < j \leq l+1$, the root $\epsilon_i - \epsilon_j$ is positive with corresponding 
root space $ \mathfrak{sl}_{\epsilon_i - \epsilon_j} =\langle E_{i,j} \rangle$ and that the root space 
corresponding to the negative of $\epsilon_i - \epsilon_j$ is generated by its transpose $E_{i,j}^t = E_{j,i}$.
Summarizing our results, we obtain that
\begin{equation*}
 \mathfrak{sl}_{l+1} = \mathfrak{h} + \sum_{ 1 \leq  i < j \leq l+1 } \langle X_{\epsilon_i - \epsilon_j}\rangle 
 + \langle X_{\epsilon_j - \epsilon_i}\rangle 
\end{equation*}  
is a Cartan decomposition for $\mathfrak{sl}_{l+1}$, where we write $X_{\epsilon_i - \epsilon_j}$ for the matrix 
$E_{i,j}$ ($i\neq j$). 
Next, we determine the co-roots. For $1 \leq i < j \leq l+1$ let the matrix $H_{i,j}$ be defined by the relation 
\begin{equation*}
 H_{i,j}= [E_{i,j},E_{j,i}]= E_{i,i}-E_{j,j} \in \mathfrak{h}.
\end{equation*}
Then, it follows from $[H_{i,j},E_{i,j}]= 2 E_{i,j}$ that $H_{i,j}$ is the co-root for $\epsilon_i - \epsilon_j$.
In the following, we denote a co-root which corresponds to a root space for a simple root $\alpha_i \in \Delta$ by 
$H_i$ ($1 \leq i \leq l$). To complete the determination of a Chevalley basis, we consider the map 
\begin{equation*}
 \phi: \mathfrak{sl}_{l+1} \rightarrow  \mathfrak{sl}_{l+1}, \ X \mapsto -X^{tr}
\end{equation*}
which is obviously an automorphism of $\mathfrak{sl}_{l+1}$ with the property that $\phi(X_{\alpha})= -X_{-\alpha}$
for $\alpha= \epsilon_i - \epsilon_j \in \Phi$. For $\alpha, \ \beta \in \Phi$, let the integer $n_{\alpha, \beta}$ 
be defined by the relation $[X_{\alpha}, X_{\beta}]=n_{\alpha, \beta} X_{\alpha + \beta}$. If we apply $\phi$
to both sides of $[X_{\alpha}, X_{\beta}]=n_{\alpha, \beta} X_{\alpha + \beta}$, we obtain 
$[X_{-\alpha}, X_{-\beta}]=-n_{\alpha, \beta} X_{-\alpha - \beta}$ from which it 
follows that $n_{-\alpha, -\beta}=-n_{\alpha, \beta}$. Since by \cite{Carter}, Theorem 4.1.2, 
$n_{\alpha, \beta} n_{-\alpha, -\beta}= -(r+1)^2$, we have that 
$n_{\alpha, \beta}= \pm (r+1)$. 
Thus, the set
 $\{ X_{\alpha}, \ H_i \mid \alpha \in \Phi, \ 1 \leq i \leq l \}$
is a Chevalley basis for $\mathfrak{sl}_{l+1}$.
\\ 
The next step is an appropriate choice of the defining matrix $A(\textit{\textbf{t}}) \in \mathfrak{sl}_{l+1}(R_1)$ 
for our parameter equation. To this purpose let
$A_{\Delta}^+:= \sum_{\alpha_i \in \Delta} X_{\alpha_i}$ be defined as in the 
previous section with the explicit matrices $X_{\alpha_i}$ from above. We set  
\begin{equation*}
A(\textit{\textbf{t}}):= A_{\Delta}^+ + \sum_{\gamma_i \in \Gamma} t_i X_{ -\gamma_i} \in \mathfrak{sl}_{l+1}(R_1).
\end{equation*}
The choice of the above Chevalley basis yields that the matrix $A(\textit{\textbf{t}})$ has the shape of 
a companion matrix, i.e. we have
\begin{equation*}
 A(\textit{\textbf{t}}) =
 \begin{pmatrix} 0 & 1 & 0 & \ldots & 0 \\
0 & 0 & 1 & & \\
\vdots& & & \ddots& \\
 0 & \ldots& & 0 & 1\\
t_1 & t_2 & \hdots & t_l & 0
\end{pmatrix},
\end{equation*}
and induces therefore the simple and nice linear differential equation  of 
Theorem~\ref{main_thm}. 
We want to mention that for similar choices of $A(\textit{\textbf{t}})$ for the groups of 
type $B_l$, $C_l$, $D_l$ and $G_2$ the equation 
$\partial(\textit{\textbf{y}})= A(\textit{\textbf{t}})\textit{\textbf{y}}$ has also a canonical cyclic vector
although $A(\textit{\textbf{t}})$ is not a companion matrix anymore. 
To complete the proof for $\mathrm{SL}_{l+1}(C)$, we need to show that the two bounds for $A(\textit{\textbf{t}})$ coincide. \\
Since we intend to apply Theorem~\ref{specialization_bound}, we need a differential equation $\partial(\textit{\textbf{y}})=\bar{A}\textit{\textbf{y}}$ 
which has $\mathrm{SL}_{l+1}(C)$ as differential Galois group
and whose defining matrix $\bar{A}$
satisfies $\bar{A} \in \mathfrak{sl}_{l+1}(R_2)$ and $\bar{A} = \sigma(A(\textit{\textbf{t}}))$ 
for a specialization $\sigma: R_1 \rightarrow R_2$. 
The following proposition shows that we have access to a large class of equations.
\begin{proposition}\label{prop_transformation}
Suppose the matrix $A$ is element of the plane $A \in A_{\Delta}^- + \mathfrak{b}^+(F) $.
Then $A$ is gauge-equivalent to a matrix in the plane
\begin{equation*}
 A_{\Delta}^- + \sum_{\gamma_i \in \Gamma} \mathfrak{sl}_{\gamma_i} (F).
\end{equation*}
\end{proposition}
 A proof of Proposition~\ref{prop_transformation} uses Lemma~\ref{lemma_cartan} and~\ref{pre_transfor_lemma} from below. 
\begin{lemma}\label{lemma_cartan}
 Suppose the matrix $A$ satisfies $A \in  A_{\Delta}^- + \mathfrak{b}^+(F)$.
 Then $A$ is gauge-equivalent to a matrix in the plane $A_{\Delta}^- +  \mathfrak{n}^+$.
\end{lemma}
\begin{proof}
The plane $A_{\Delta}^- + \mathfrak{b}^+(F)$ writes as 
\begin{equation*}
 A_{\Delta}^- + \mathfrak{h}(F)+ \mathfrak{n}^+(F).
\end{equation*}
Thus, we need to show that we can delete by a gauge transformation the components of $A$ 
which lie in the Cartan subalgebra $\mathfrak{h}(F)$.
For  $j \in \{ 1, \dots , l+1 \}$ we denote by $A_j$ a matrix of the plane
 \begin{equation*}
  A_j \in  A_{\Delta}^- + \sum_{j \leq i \leq l} \mathfrak{h}_i + \mathfrak{n}^+(F),  
 \end{equation*}
 where $\mathfrak{h}_i = \langle H_i \rangle$. Note that we have $A_1 \in A_{\Delta}^- + \mathfrak{b}^+(F)$ and $A_{l+1} \in A_{\Delta}^- + \mathfrak{n}^+(F)$.
 We prove now the following assertion:
 Any matrix $A_j$ is gauge-equivalent to a matrix $A_{j+1}$ of the corresponding plane for $1 \leq j \leq l$. 
 To this purpose let a matrix $A_j$ be given by
  \begin{equation*}
   A_j = A_{\Delta}^- + \sum_{j \leq i \leq l} h_i H_i + \sum_{\alpha \in \Phi^+} Z_{\alpha}   
 \end{equation*}
 with $ h_i \in F$ and $Z_{\alpha} \in \mathfrak{g}_{\alpha}$.
We compute the gauge transformation of $A_j$ by a parametrized root group element $u_{\alpha_j}(x)=\mathrm{exp}(x X_{\alpha_j})$ and show that we
can choose $x \in F$ such that the coefficient of $H_j$ in 
\begin{equation*}
 \mathrm{Ad}(u_{\alpha_j}(x)) (A_j) + l \delta(u_{\alpha_j}(x)) 
 \end{equation*}
 vanishes. First, we compute the image of $A_j$ under the adjoint action and afterwards, 
 we determine the image of $\mathrm{exp}(x X_{-\alpha_j})$ under 
 the logarithmic derivate. Since the adjoint action 
 is linear, we can consider each summand of $A_j$ separately. We start with the computation of $\mathrm{Ad}(u_{-\alpha_j}(x)) (A_{\Delta}^-)$.
Remark~\ref{remark3} yields that 
$\mathrm{Ad}(u_{\alpha_j}(x)) (X_{-\alpha_j})=X_{-\alpha_j} + x H_{j} -x^2 X_{\alpha_j}$ and 
$\mathrm{Ad}(u_{\alpha_j}(x)) (X_{-\alpha_i})=X_{-\alpha_i}$ for $1\leq i \leq l$ and $i \neq j$. 
Summing up, we obtain
\begin{equation*}
 \mathrm{Ad}(u_{\alpha_j}(x)) (A_{\Delta}^-)= \sum_{i=1}^l  \mathrm{Ad}(u_{\alpha_j}(x)) (X_{\alpha_i}) \in A_{\Delta}^- + x H_j  + 
\mathfrak{sl}_{\alpha_j} (F).
 \end{equation*}
By Remark~\ref{remark3}, we have that $\mathrm{Ad}(u_{\alpha_j}(x)) (H_i)= H_i - \langle \alpha_j, \alpha_i\rangle X_{\alpha_j}$ 
for $j \leq i \leq l$ from which we deduce that
\begin{equation*}
 \mathrm{Ad}(u_{\alpha_j}(x)) (\sum_{j \leq i \leq l} h_i H_i) \in \sum_{j \leq i \leq l} h_i H_i  
+  \mathfrak{sl}_{\alpha_j}(F).
\end{equation*}
Since the subspace $\mathfrak{n}^+$  
is stabilized by $\mathrm{Ad}(u_{\alpha_j}(x))$,
 we obtain that   
\begin{equation*}
 \mathrm{Ad}(u_{\alpha_j}(x)) (\sum_{\alpha \in \Phi^+} Z_{\alpha}) \in \mathfrak{n}^+. 
\end{equation*}
Finally, Remark~\ref{remark4} implies that $l\delta (u_{\alpha_j}(x)) \in \mathfrak{sl}_{\alpha_j}(F)$.
Summing up our results, we get that 
 \begin{equation*}
  \mathrm{Ad}(u_{\alpha_j}(x))(A_j) + \l\delta(u_{\alpha_j}(x)) \in A_{\Delta}^- + (x+ h_j)H_j +  \sum_{ j+1 \leq i \leq l} h_i H_i + 
\sum_{\alpha \in \Phi^+}  \tilde{Z}_{\alpha}  
 \end{equation*}
 with suitable $\tilde{Z}_{\alpha} \in \mathfrak{sl}_{\alpha}(F)$.
 It follows that $A_j$ is gauge-equivalent to matrix of shape $A_{j+1}$ for $x=-h_j$.\\
 We show now by induction that for all  $j \in \{1, \dots, l \}$ the following assertion holds: The matrix $A$ is gauge-equivalent to a matrix 
  \begin{equation*}
 A_{j+1} \in  A_{\Delta}^- + \sum_{j+1 \leq i \leq l} \mathfrak{h}_i + \mathfrak{n}^+(F).
 \end{equation*} 
 The above argument shows that $A$ is gauge
 equivalent to a matrix $A_2$ of the required shape, i.e. the assumption is shown for $j=1$. 
 Assume $j>1$. Then the induction assumption yields that $A$ is gauge-equivalent to 
 \begin{equation*}
 A_j \in  A_{\Delta}^- + \sum_{j \leq i \leq l} \mathfrak{h}_i + \mathfrak{n}^+(F).
 \end{equation*} 
But then the above argument applied to $A_j$ shows that $A_j$ is gauge-equivalent 
to a matrix $A_{j+1}$ in the required plane. 
Thus, $A$ is gauge-equivalent to the matrix $A_{j+1}$ and the induction 
assumption holds for all $j \in \{1, \dots, l \}$. \\
The assertion of the lemma follows now from the case $j=l$.
\end{proof}
\begin{lemma}\label{pre_transfor_lemma}
 Let $k \in \{ 1, \dots ,l \}$ and suppose the matrix $A$ satisfies
 \begin{equation*}
  A \in A_{\Delta}^- + \sum_{ \gamma_i \in \Gamma_k} \mathfrak{sl}_{\gamma_i} (F)
  +  \sum_{ \alpha \in \Phi_k^+} \mathfrak{sl}_{\alpha}(F).
 \end{equation*}
Then $A$ is gauge-equivalent to a matrix in the plane
\begin{equation*}
 A_{\Delta}^- + \sum_{\gamma_i \in \Gamma_{k-1}} \mathfrak{sl}_{\gamma_i}(F)
 +  \sum_{ \alpha \in \Phi_{k-1}^+} \mathfrak{sl}_{\alpha}(F),
\end{equation*}
where we recall that $\Phi^+_0$ is the empty set and $\Gamma_0 = \Gamma$.
\end{lemma}
\begin{proof}
 For $1 \leq j \leq k$ we denote in the following by $\Phi_{k,j}$ the set of all roots $\alpha \in \Phi^+$ which satisfy 
$\alpha \in \Phi_k^+ \setminus \Phi_{k-1}^+$ and $ \mathrm{ht}(\alpha)\geq j$. Since $ \Phi_{k,1} =\Phi_k^+ \setminus \Phi_{k-1}^+ $, 
we obtain by Remark~\ref{remark1} that 
\begin{equation*}
 A_{\Delta}^- + \sum_{\gamma_i \in \Gamma_k} \mathfrak{sl}_{\gamma_i} + \sum_{ \alpha \in \Phi_{k}^+} \mathfrak{sl}_{\alpha}(F) 
 = A_{\Delta}^- + \sum_{\gamma_i \in \Gamma_k} \mathfrak{sl}_{\gamma_i} + \sum_{ \alpha \in \Phi_{k,1}} \mathfrak{sl}_{\alpha}(F) 
 +  \sum_{ \alpha \in \Phi_{k-1}^+} \mathfrak{sl}_{\alpha}(F).
\end{equation*}
We will prove inductively that we can delete by a gauge transformation the components of $A$  
which lie in the root spaces 
$\mathfrak{sl}_{\alpha}(F)$ for all $\alpha \in \Phi_{k,1}$ except for that $\alpha \in \Phi_{k,1}$ which satisfies $\mathrm{ht}(\alpha)=k$.
Then, since $\alpha \in \Phi_{k,1}$ with $\mathrm{ht}(\alpha)=k$ is the root $\gamma_k$ and $\Gamma_{k-1}\setminus \Gamma_{k}= \{ \gamma_k\}$, 
the assertion of the lemma follows if we rearrange the summads in the above equation accordingly.\\
In the following, we denote for $j \in \{ 1, \dots , k-1 \}$ by $A_j$ a matrix of the plane 
 \begin{equation*}
  A_j \in A_{\Delta}^- + \sum_{\gamma_i \in \Gamma_k} \mathfrak{sl}_{ \gamma_i} + \sum_{ \alpha \in \Phi_{k,j}} \mathfrak{sl}_{\alpha} + 
  \sum_{ \alpha \in \Phi_{k-1}^+} \mathfrak{sl}_{\alpha}.
 \end{equation*}
We show the following assertion: For $j \in \{ 1, \dots , k-1 \}$ a matrix $A_j$ is gauge-equivalent to a matrix
 \begin{equation*}
  A_{j+1} \in A_{\Delta}^- + \sum_{\gamma_i \in \Gamma_k} \mathfrak{sl}_{ \gamma_i} + \sum_{ \alpha \in \Phi_{k,j+1}} \mathfrak{sl}_{\alpha} + 
  \sum_{ \alpha \in \Phi_{k-1}^+} \mathfrak{sl}_{\alpha}.
 \end{equation*}
Let  
 \begin{equation*}
  A_j = A_{\Delta}^- + \sum_{\gamma_i \in \Gamma_k} Z_{ \gamma_i} + \sum_{ \alpha \in \Phi_{k,j}} Z_{\alpha} + 
  \sum_{ \alpha \in \Phi_{k-1}^+} Z_{\alpha},
 \end{equation*}
 where for $\alpha \in \Phi^+$, $Z_{\alpha}$ denotes an element of the root space $\mathfrak{sl}_{\alpha}(F)$.
 The gauge transformation of $A_j$ will be done with a root group element $u_{\beta}(x)=\mathrm{exp}(xX_{\beta})$ and, as in Lemma~\ref{lemma_cartan}, 
 we compute the images $\mathrm{Ad}(u_{\beta}(x))(A_j)$ and $l \delta (u_{\beta}(x))$ separately.
 We have to show that we can remove the component $Z_{\alpha}$ of $A$ 
 where $\alpha$ satisfies $\alpha \in \Phi_{k,j}$ 
 and $\mathrm{ht}(\alpha)=j$. By Remark~\ref{remark2} a root with this property is unique 
 and we denote it by $\bar{\alpha}$. 
 From Remark~\ref{remark2}  we also know that there is a unique simple root $\alpha_s \in \Delta$ such that 
 $\bar{\alpha} + \alpha_s \in \Phi_k^+ \setminus \Phi_{k-1}^+$. 
Thus, for $\beta := \bar{\alpha} + \alpha_s$ we obtain $\beta- \alpha_s = \bar{\alpha}$ and 
if for $\alpha_t \in \Delta$ with $\alpha_t \neq \alpha_s$ the sum $\beta - \alpha_t$ is a root, then $\beta - \alpha_t $ is an element of $ \Phi_{k-1}^+$.
These arguments and Remark~\ref{remark3} show that
\begin{equation*}
\mathrm{Ad}(u_{\beta}(x))(X_{-\alpha_s}) = X_{-\alpha_s} + x \ c_{\beta, -\alpha_s,1} X_{\bar{\alpha}} 
\end{equation*}
and $\mathrm{Ad}(u_{\beta}(x))(X_{-\alpha_i}) \in \sum_{ \alpha \in \Phi_{k-1}^+}  \mathfrak{sl}_{\alpha}(F)$ 
for $1\leq i \leq l$ and $i \neq s$. Hence, we obtain
\begin{equation*}
 \mathrm{Ad}(u_{\beta}(x))(A_{\Delta}^-) \in  A_{\Delta}^- + x \ c_{\beta, -\alpha_s,1} X_{\bar{\alpha}}
  + \sum_{ \alpha \in \Phi_{k-1}^+}  \mathfrak{sl}_{\alpha}(F).
\end{equation*}
Next, we compute the image of $\sum_{ \gamma_i \in \Gamma_{k}} Z_{\gamma_i}$ under $\mathrm{Ad}(u_{\beta}(x))$.
Since $\gamma_i$ is the maximal root of $\Phi_{i}^+$ for $k+1 \leq i \leq l$, at least one of the coefficients in $\gamma_i + q \beta$ 
must be greater than $1$ for $q \in \mathds{N} \setminus \{ 0 \}$ and, therefore, $\gamma_i + q \beta$ can not be a root of $\Phi$. This shows that the subspace 
$\sum_{\gamma_i \in \Gamma_{k}} \mathfrak{sl}_{\gamma_i}(F)$ is invariant under
$\mathrm{Ad}(u_{\beta}(x))$. In other words we have 
\begin{equation*}
  \mathrm{Ad}(u_{\beta}(x))(\sum_{ \gamma_i \in \Gamma_{k}} Z_{\gamma_i}) = \sum_{\gamma_i \in \Gamma_{k}} Z_{\gamma_i}.
\end{equation*}
A similar argumentation yields 
\begin{equation*}
 \mathrm{Ad}(u_{\beta}(x))(\sum_{ \alpha \in \Phi_{k,j} } Z_{\alpha}) =
  Z_{\bar{\alpha}} + \sum_{ \alpha \in \Phi_{k,j+1} } Z_{\alpha} =
  c_{\bar{\alpha}} X_{\bar{\alpha}} + \sum_{ \alpha \in \Phi_{k,j+1} } Z_{\alpha},
\end{equation*}
where $c_{\bar{\alpha}} \in F$ such that $Z_{\bar{\alpha}}=c_{\bar{\alpha}} X_{\bar{\alpha}}$.\\
We determine now the image $\mathrm{Ad}(u_{\beta}(x))( Z_{\alpha})$ for $\alpha \in \Phi_{k-1}^+$. 
If for $q \in \mathds{N} \setminus \{ 0 \}$ 
the sum $\alpha + q \beta$ is a root of $\Phi$, 
then it is an element of $\Phi_{k, j+2}$, since $\beta \in \Phi_{k, j+1}$ and $\mathrm{ht}(\beta)= j+1$. We conclude by Remark~\ref{remark3} that 
\begin{equation*}
 \mathrm{Ad}(u_{\beta}(x))( \sum_{ \alpha \in \Phi_{k-1}^+} Z_{\alpha}) \in \sum_{ \alpha \in \Phi_{k-1}^+} Z_{\alpha}
 + \sum_{ \alpha \in \Phi_{k,j+2}} \mathfrak{sl}_{\alpha}(F).
\end{equation*}
Finally, the logarithmic derivate of $u_{\beta}(x)$ lies in the root space 
$\mathfrak{sl}_{\beta}(F)$ by Remark~\ref{remark4}. 
If we sum up everything from above, we get for suitable $\tilde{Z}_{\alpha} \in \mathfrak{sl}_{\alpha}(F)$ that
\begin{gather*}
 \mathrm{Ad}(u_{\beta}(x))(A)+ l\delta  (u_{\beta}(x)) = \\
 A_{\Delta}^- + (c_{\beta, -\alpha_s,1}x + c_{\bar{\alpha}})X_{\bar{\alpha}} +  \sum_{\gamma_i \in \Gamma_k} Z_{\gamma_i} 
 + \sum_{  \alpha \in \Phi_{k,j+1} } \tilde{Z}_{\alpha} + \sum_{ \alpha \in \Phi_{k-1}^+} \tilde{Z}_{\alpha}.
 \end{gather*}
Choosing $x= - c_{\bar{\alpha}}c_{\beta, -\alpha_s,1}^{-1}$, we obtain that $A_j$ is gauge-equivalent to a matrix of shape $A_{j+1}$.\\
We show now inductively on $j \in \{1, \dots, k-1\}$ that the matrix $A$ is gauge-equivalent to a matrix 
\begin{equation*}
  A_{j+1} = A_{\Delta}^- + \sum_{\gamma_i \in \Gamma_k} Z_{ \gamma_i} + \sum_{ \alpha \in \Phi_{k,j+1}} Z_{\alpha} + 
  \sum_{ \alpha \in \Phi_{k-1}^+} Z_{\alpha}.
\end{equation*}
If $j=1$, then $A$ is clearly gauge-equivalent to a matrix $A_2$ by the above argumentation what shows the case $j=1$. 
For $j>1$ we obtain by the induction assumption that $A$ is gauge-equivalent to 
\begin{equation*}
  A_{j} = A_{\Delta}^- + \sum_{\gamma_i \in \Gamma_k} Z_{ \gamma_i} + \sum_{ \alpha \in \Phi_{k,j}} Z_{\alpha} + 
  \sum_{ \alpha \in \Phi_{k-1}^+} Z_{\alpha}.
\end{equation*} 
Again, the above argument applied to $A_j$ yields that $A_j$ is gauge-equivalent to a matrix of form $A_{j+1}$ 
and, therefore, $A$ is gauge-equivalent to $A_{j+1}$. Finally, the induction yields for $j=k-1$ that $A$ is gauge-equivalent to
 \begin{equation*}
  A_{k} = A_{\Delta}^- + \sum_{\gamma_i \in \Gamma_k} Z_{ \gamma_i} + \sum_{ \alpha \in \Phi_{k,k}} \tilde{Z}_{\alpha} + 
  \sum_{ \alpha \in \Phi_{k-1}^+} \tilde{Z}_{\alpha},
 \end{equation*}
 where $\Phi_{k,k}$ consists of the single element $\gamma_k$. Hence, after rearranging the summands in the above equation, 
 the assertion of the lemma follows.
\end{proof}
\begin{proof_prop_trans}
Let $A \in  A_{\Delta}^- + \mathfrak{b}^+(F)$. Then Lemma~\ref{lemma_cartan} implies that $A$ is gauge-equivalent to a matrix 
$A_1$ in the plane
 \begin{equation*}
 A_{\Delta}^- + \sum_{\alpha \in \Phi^+} \mathfrak{sl}_{\alpha} (F)= 
 A_{\Delta}^- + \sum_{\gamma_i \in \Gamma_l} \mathfrak{sl}_{\gamma_i} (F) + \sum_{\alpha \in \Phi_l^+ } \mathfrak{sl}_{\alpha}(F).
\end{equation*}
Note that $\Gamma_l= \emptyset$ and $\Phi_l^+= \Phi^+$ and, therefore, equality between the two planes holds.
We make now the following inductive assumption on the integer $j \in \{ 1, \dots, l-1 \}$: The matrix $A_1$ is gauge-equivalent to 
a matrix
 \begin{equation*}
 A_{j+1} \in  A_{\Delta}^- + \sum_{\gamma_i \in \Gamma_{l-j}} \mathfrak{sl}_{\gamma_i} (F) + \sum_{\alpha \in \Phi_{l-j}^+ } \mathfrak{sl}_{\alpha}(F).
\end{equation*}
If we choose $k=l$ in Lemma~\ref{pre_transfor_lemma}, we obtain that $A_1$ is gauge-equivalent to a matrix $A_2$ of the required shape. 
Thus, the induction assumption is shown in 
case of $j=1$. Now let $j > 1$. From the induction assumption for $j-1$ we obtain that $A_1$ is gauge-equivalent to 
 \begin{equation*}
 A_{j} \in  A_{\Delta}^- + \sum_{\gamma_i \in \Gamma_{l-j+1}} \mathfrak{sl}_{\gamma_i} (F) + \sum_{\alpha \in \Phi_{l-j+1}^+ } \mathfrak{sl}_{\alpha}(F).
\end{equation*}
Then, the assertion of Lemma~\ref{pre_transfor_lemma} for $k=l-j+1$ yields that $A_{j}$ is gauge-equivalent to a matrix of shape $A_{j+1}$. 
This proves the induction assumption.\\
Thus, for $j=l-1$ we get a gauge equivalence between $A$ and a matrix in the plane 
 \begin{equation*}
 A_{\Delta}^- + \sum_{\gamma_i \in \Gamma_{1}} \mathfrak{sl}_{\gamma_i} (F) + \sum_{\alpha \in \Phi_{1}^+ } \mathfrak{sl}_{\alpha}(F)= 
 A_{\Delta}^- + \sum_{\gamma_i \in \Gamma} \mathfrak{sl}_{\gamma_i} (F),
\end{equation*}
where the equality follows from the fact that $\Phi_{1}^+ = \{ \alpha_l \}$ and $\Gamma \setminus \Gamma_{1}=\{ \alpha_l \}$.
\end{proof_prop_trans}
\begin{proof_thm1}
 We define a differential equation $\partial(\textit{\textbf{y}})= A(\textit{\textbf{t}}) \textit{\textbf{y}}$ by
 \begin{equation*}
  A(\textit{\textbf{t}})=A_{\Delta}^+ + \sum_ {\gamma_i \in \Gamma } t_i X_{-\gamma_i} \in \mathfrak{sl}(R_1).
 \end{equation*}
Since $A(\textit{\textbf{t}}) \in \mathfrak{sl}(R_1)$, Proposition~\ref{Prop_upper} shows that the differential Galois group $G(C)$ of 
$\partial(\textit{\textbf{y}})= A(\textit{\textbf{t}}) \textit{\textbf{y}}$ is a subgroup of $\mathrm{SL}_{l+1}(C)$.
On the other hand, by Corollary~\ref{mod_Mitch_Singer} there exists $A_1 \in \mathfrak{h}(C)$ such that
the differential Galois group of 
$\partial(\textit{\textbf{y}})=(A_0 + z A_1)\textit{\textbf{y}}$ over $F_2$ is $\mathrm{SL}_{l+1}(C)$, where $A_0$ is as in 
Proposition~\ref{mod_Mitch_Singer}, and $(A_0 + z A_1)$ is by its construction an element of the plane 
$A_{\Delta}^+ +  \mathfrak{b}^-$.
Now, if we interchange the role of the positive and negative roots in Proposition~\ref{prop_transformation}, we obtain that $(A_0 + z A_1)$ 
is gauge-equivalent to a matrix $\bar{A}$ in the plane 
\begin{equation*}
A_{\Delta}^+ + \sum_{\gamma_i \in \Gamma} \mathfrak{sl}_{-\gamma_i}(R_2).
\end{equation*}
 For the specialization $\sigma:R_1\rightarrow R_2$, $\textit{\textbf{t}} \mapsto (f_1, \dots , f_l)$,
where $f_i \in F$ such that  
\begin{equation*}
 \bar{A} = A_{\Delta} + \sum_{-\gamma_i \in \Gamma^-}  f_i X_{-\gamma_i},
\end{equation*}
we get that the differential Galois 
group of $\partial(\textit{\textbf{y}})=A(\sigma(\textit{\textbf{t}}))\textit{\textbf{y}}$ over $F_2$ is $\mathrm{SL}_{l+1}(C)$. Then by 
Theorem~\ref{specialization_bound} we have $\mathrm{SL}_{l+1}(C) \subseteq G(C)$ and in combination with the above relation we obtain that
$G(C) = \mathrm{SL}_{l+1}(C)$. \\
The matrix $A(t)$ is a companion matrix with trace zero. It follows that  $\partial(\textit{\textbf{y}})=A(\textit{\textbf{t}})\textit{\textbf{y}}$
is equivalent to the linear  differential equation in Theorem~\ref{main_thm}.
\end{proof_thm1}

\section{Generic properties of the parameter equation for $\mathrm{SL}_{l+1}$}
The following proposition is a refined version of Theorem~\ref{thm_gen} from the introduction. For a proof 
of Theorem~\ref{thm_gen} see Remark~\ref{remark_gen} at the end of this section.
\begin{proposition}\label{generic_sl}
 Let $F$ be a differential field with field of constants $C$ and suppose $F$ satisfies the 
following property: For all $f \in F$, $F$ contains all $(l+1)$-roots of $f$.\\
Let $E/F$ be a Picard-Vessiot
 extension with defining matrix $A \in \mathrm{M}_{l+1}(F)$ and differential Galois group 
 $H(C) \subseteq \mathrm{SL}_{l+1}(C)$. Then there exists
 a specialization $\sigma : R_1 \rightarrow F$ such that
 \begin{equation*}
L(y,\sigma(\textit{\textbf{t}}))= y^{(l+1)} - \sum\nolimits_{i=1}^{l} \sigma(t_i) y^{(i-1)}=0 
\end{equation*}
defines a Picard-Vessiot extension which is differentially isomorphic to $E/F$.
\end{proposition}
\begin{proof}
By the Cyclic Vector Theorem (see for instance \cite{Kovcyc}, page 3), we can assume that the defining 
matrix  $A$ is a companion matrix, i.e. for $a_i \in F$, 
$A$ has shape 
\begin{equation*}
A= \begin{pmatrix} 0 & 1 & 0 & \ldots & 0 \\
0 & 0 & 1 & & \\
\vdots& & & \ddots& \\
 0 & \ldots& & 0 & 1\\
a_1 & a_2 & \hdots &  & a_{l+1}
\end{pmatrix}.
\end{equation*}
Then a fundamental solution matrix for the equation $\partial(\textit{\textbf{y}})=A\textit{\textbf{y}}$ is a Wronskian matrix
$W(y_1, \dots, y_{l+1})=:Y \in \mathrm{GL}_{l+1}(E)$ and an inductive argument shows that
the coefficient $a_{l+1}$ of $A$ satisfies 
\begin{equation*}
 a_{l+1}=\frac{\partial(\mathrm{det}(Y))}{\mathrm{det}(Y)}.
\end{equation*}
Now, for $C \in H(C) \subseteq \mathrm{SL}_{l+1}(C)$ it follows that
$\mathrm{det}(YC) =\mathrm{det}(Y)$, meaning that    
 $\mathrm{det}(Y)$ is invariant under the action of the differential Galois group. We conclude that
$\mathrm{det}(Y) =:f$ is an element of $F$.\\
We show that the equation $\partial(\textit{\textbf{y}})= A\textit{\textbf{y}}$ is gauge-equivalent 
to a differential equation $\partial(\textit{\textbf{y}})=\bar{A}\textit{\textbf{y}}$
with defining matrix
\begin{equation*}
\bar{A}= \begin{pmatrix} 0 & 1 & 0 & \hdots & 0 \\
0 & 0 & 1 & & \\
\vdots& & & \ddots& \\
 0 & \ldots& & 0 & 1\\
f_1& f_2 & \hdots &  f_l & 0
\end{pmatrix}.
\end{equation*}
The first step is to show that $\bar{A}$ is gauge-equivalent to an element of the Lie algebra 
of $\mathrm{SL}_{l+1}(F)$. To this purpose,
let $B_1$ be the diagonal matrix $B_1:= \mathrm{diag}(1,\dots,1, \frac{1}{f})$. 
Simple matrix multiplications show that
\begin{equation*}
\mathrm{Ad} (B_1) (A) + l\delta (B_1) = \begin{pmatrix} 0 & 1 & 0 & \hdots & 0 \\
0 & 0 & \ddots & & \\
\vdots& & & 1 & 0 \\
 0 & \ldots& & 0 & f\\
\frac{a_1}{f} & \frac{a_2}{f} & \hdots & \frac{a_{l}}{f} & 0
\end{pmatrix}:=A_1 \in \mathfrak{sl}_{l+1}(F) .
\end{equation*}
Note that until now we did not need any additional assumptions on the differential field $F$. 
We conclude that for every differential field $F$ and every Picard-Vessiot ring $R/F$ 
with differential Galois group $\mathrm{SL}_{l+1}(C)$ the corresponding torsor $\mathcal{Z}$
has an $F$-rational point.\\
The next step is a gauge transformation of $A_1$ by the diagonal matrix 
\begin{equation*}
B_2:=\mathrm{diag}((\frac{1}{f})^{\frac{1}{l+1}},\dots,(\frac{1}{f})^{\frac{1}{l+1}},(\frac{1}{f})^{-\frac{l}{l+1}})
\end{equation*}
where we assumed that the differential field $F$ contains an $(l+1)$-root of $f$. 
We obtain
\begin{equation*}
\mathrm{Ad} (B_2) (A_1)  + l \delta(B_2) = 
\begin{pmatrix} 
- \frac{\partial(f)}{(l+1) f} & 1 & 0 & \hdots & 0 \\
0 & - \frac{\partial(f)}{(l+1)  f} & \ddots & & \\
\vdots& & & 1  & 0 \\
 0 & \ldots& & - \frac{\partial(f)}{(l+1) f} & 1\\
a_1 & a_2 & \hdots & a_{l} & \frac{l \, \partial(f)}{(l+1) f}
\end{pmatrix}
=: A_2 .
\end{equation*}
Since $A_2$ is an element of the plane
\begin{equation*}
 A_{\Delta}^+ + \mathfrak{h}(F) + \sum_{\alpha \in \Phi^-} \mathfrak{sl}_{\alpha}(F),
\end{equation*}
we can apply Proposition~\ref{prop_transformation}. It yields that there exists a matrix
$B_3 \in \mathrm{SL}_{l+1}(F)$ such that 
\begin{equation*}
\mathrm{Ad}(B_3) (A_2) + l\delta (B_3) = \begin{pmatrix} 0 & 1 & 0 & \hdots & 0 \\
0 & 0 & 1 & & \\
\vdots& & & \ddots& \\
 0 & \ldots& & 0 & 1\\
f_1& f_2 & \hdots &  f_l & 0
\end{pmatrix}=: \bar{A}.
\end{equation*}
Hence, the matrix $B_3 B_2 B_1 \in \mathrm{GL}_{l+1}(F)$ defines a differential isomorphism
from $E/F$ to a Picard-Vessiot extension $\tilde{E}/F$ for the differential equation
$\partial(\textit{\textbf{y}})= \bar{A} \textit{\textbf{y}}$ and the map 
$\sigma: (t_1,\dots,t_l) \mapsto (f_1,\dots,f_l)$ is the required specialization.
\end{proof}

\begin{remark}\label{remark_gen}
Note that an algebraically closed 
differential field $\bar{F}$ with field of constants $C$ satisfies automatically the condition in Proposition~\ref{generic_sl}.
Further, if $F$ is an arbitrary differential field with constants $C$ and 
$\sigma: R_1 \rightarrow F$ is any 
specialization of the parameters, then by Proposition~\ref{Prop_upper} the differential Galois group
of a Picard-Vessiot extension for $L(y,\sigma(\textit{\textbf{t}}))$ is a subgroup of $\mathrm{SL}_{l+1}(C)$.
This proves Theorem \ref{thm_gen} from the introduction.
\end{remark}

\section{Further results and conclusions}
As outlined in the Section 2 and 3, it is possible to apply our method to the remaining classical groups.
For example, for the groups of type $B_l$, $C_l$, $D_l$ and $G_2$ (here $l=2$) we proved in \cite{Seiss} that similar nice linear parameter
differential equations can be computed. The results for these groups are summarized in the following theorem. 
\begin{theorem}
\label{main_thm_conclusion}
The linear parameter differential equation
\begin{enumerate}
\item 
 $ \begin{aligned}[t] 
L(y,\textbf{t})= y^{(l+1)} - \sum\nolimits_{i=1}^{l} t_i y^{(i-1)}=0
\end{aligned}$ 
has $\mathrm{SL}_{l+1}(C)$ as differential Galois group over $F_1$.
\vspace{2mm}
\item $\begin{aligned}[t] 
L(y,\textbf{t})= y^{(2l)} - \sum\nolimits_{i=1}^{l} (-1)^{i-1} (t_i y^{(l-i)})^{(l-i)}=0
\end{aligned}$  
has $\mathrm{SP}_{2l}(C)$ as differential Galois group over $F_1$.
\vspace{2mm}
\item $\begin{aligned}[t] 
L(y,\textbf{t})= y^{(2l+1)} - \sum\nolimits_{i=1}^{l} (-1)^{i-1} ((t_i y^{(l+1-i)})^{(l-i)}+(t_i y^{(l-i)})^{(l+1-i)})= 0
\end{aligned}$\\
has $\mathrm{SO}_{2l+1}(C)$ as differential Galois group over $F_1$.
\vspace{2mm}
\item $\begin{aligned}[t] 
L(y,\textbf{t})=  y^{(2l)} -2 \sum\nolimits_{i=3}^{l} (-1)^{i} ((t_i y^{(l-i)})^{(l+2-i)}+(t_i y^{(l+1-i)})^{(l+1-i)})-
\end{aligned}$ \\
$\begin{aligned}[t]
(t_2 y^{(l-2)} + t_1 y)^{(l)} - ((-1)^l t_1 z_1 + z_2) -\sum\nolimits_{i=0}^{l-2} (t_2^{l-2-i} z_1)^{(i)} =0
\end{aligned}$ 
has $\mathrm{SO}_{2l}(C)$ as differential Galois group over $F_1$. 
\vspace{2mm}
\item $\begin{aligned}[t] 
L(y,\textbf{t})= y^{(7)} + 2t_1 y' + 2 (t_1 y)' + 2 (t_2 y^{(4)})' + (t_2y')^{(4)} -2 (t_2(t_2y')')'= 0
\end{aligned}$ \\
has $\mathrm{G}_{2}(C)$ as differential Galois group over $F_1$.
\end{enumerate}
The substitutions $z_1$ and $z_2$ in (4) are given by
\begin{eqnarray*}
z_1 &:=& y^{(l)}- t_2 y^{(l-2)} - t_1 y \\ 
z_2 &:=& \frac{(t_2^{(l-2)}+(-1)^{l-2}t_1)^{(1)}}{t_2^{(l-2)}+(-1)^{l-2}t_1} \cdot \bigg(  y^{(2l-1)}- 
(t_2 y^{(l-2)} + t_1 y)^{(l-1)}
 \\
&& -  2 \sum_{i=3}^{l} (-1)^{i} ( (t_i y^{(l-i)})^{(l+1-i)} + (t_i y^{(l+1-i)})^{(l-i)} )- \sum_{i=0}^{l-3} (t_2^{(l-3-i)} z_1)^{(i)} \bigg).
\end{eqnarray*}
\end{theorem}
The parameter differential equations in Theorem~\ref{main_thm_conclusion} define large families of linear 
differential equations. By construction they represent well the geometric structure of the underlying Lie groups 
(see \cite{Kostant} for an interesting connection) and they seem to be very general. 
So far we do not know 
which types of Picard-Vessiot extensions $E/F$ can be obtained by specializations of the parameters. 
Since the defining matrices are elements in the Lie algebras of the corresponding groups, we have to restrict 
this question to Picard-Vessiot extensions which are function fields of the trivial torsor. 
The best we can hope for is that our equations are quasi-generic equations.


\end{document}